\documentclass{amsart}
\usepackage{amsmath,amsthm,amssymb,hyperref,mathrsfs}
\usepackage{aliascnt}
\usepackage{lmodern}
\usepackage[T1]{fontenc}

\usepackage{xcolor}
\definecolor{dblue}{rgb}{0,0,0.70}
\hypersetup{
	unicode=true,
	colorlinks=true,
	citecolor=dblue,
	linkcolor=dblue,
	anchorcolor=dblue
}

\makeatletter
\expandafter\g@addto@macro\csname th@plain\endcsname{%
		\thm@notefont{\bfseries}
	}%
\expandafter\g@addto@macro\csname th@remark\endcsname{%
		\thm@headfont{\bfseries}
	}%
\makeatother

\newtheorem{theorem}{Theorem}[section]	
\newtheorem*{theorem*}{Theorem}

\newaliascnt{lemma}{theorem}
\newtheorem{lemma}[lemma]{Lemma}
\aliascntresetthe{lemma}
\newtheorem*{lemma*}{Lemma}

\newaliascnt{proposition}{theorem}
\newtheorem{proposition}[proposition]{Proposition}
\aliascntresetthe{proposition}

\newaliascnt{corollary}{theorem}
\newtheorem{corollary}[corollary]{Corollary}
\aliascntresetthe{corollary}

\theoremstyle{remark}

\newaliascnt{remark}{theorem}
\newtheorem{remark}[remark]{Remark}
\aliascntresetthe{remark}
\newaliascnt{question}{theorem}
\newtheorem{question}[question]{Question}
\aliascntresetthe{question}

\newaliascnt{definition}{theorem}
\newtheorem{definition}[definition]{Definition}
\aliascntresetthe{definition}

\newaliascnt{example}{theorem}

\aliascntresetthe{example}

\renewcommand{\restriction}{\mathbin\upharpoonright}

\newcommand{\axiom}[1]{\mathsf{#1}} 
\newcommand{\ZFC}{\axiom{ZFC}}

\newcommand{\AD}{\axiom{AD}}
\newcommand{\DC}{\axiom{DC}}
\newcommand{\ZF}{\axiom{ZF}}

\newcommand{\Ord}{\mathrm{Ord}}

\newcommand{\GCH}{\axiom{GCH}}

\newcommand{\HS}{\axiom{HS}}

\DeclareMathOperator{\cf}{cf}
\DeclareMathOperator{\dom}{dom}

\DeclareMathOperator{\supp}{supp}

\DeclareMathOperator{\sym}{sym}
\DeclareMathOperator{\fix}{fix}

\DeclareMathOperator{\id}{id}
\DeclareMathOperator{\aut}{Aut}
\DeclareMathOperator{\Col}{Col}
\DeclareMathOperator{\Add}{Add}

\DeclareMathOperator{\otp}{otp}

\newcommand{\forces}{\mathrel{\Vdash}}

\newcommand{\PP}{\mathbb{P}}
\newcommand{\power}{\mathcal{P}}

\newcommand{\QQ}{\mathbb{Q}}

\newcommand{\cF}{\mathcal F}
\newcommand{\cU}{\mathcal U}
\newcommand{\sF}{\mathscr F}

\newcommand{\cS}{\mathcal S}
\newcommand{\sG}{\mathscr G}

\newcommand{\tup}[1]{\langle#1\rangle}

\newcommand{\middd}{\mathrel{}\middle|\mathrel{}}

\author{Yair Hayut}
\author{Asaf Karagila}
\thanks{The second author was partially supported by the Austrian Science Foundation FWF, grant I~3081-N35}
\address[Yair Hayut]{School of Mathematical Sciences.
Tel Aviv University.
Tel Aviv 69978,
Israel}
\email[Yair Hayut]{yair.hayut@mail.huji.ac.il}
\address[Asaf Karagila]{DMG/Algebra, TU Wien.
Wiedner Hauptstrasse 8-10/104,
1040 Wien,
Austria}
\curraddr[Asaf Karagila]{School of Mathematics,
University of East Anglia.
Norwich, NR4~7TJ, United Kingdom
}
\email[Asaf Karagila]{karagila@math.huji.ac.il}
\urladdr[Asaf Karagila]{http://karagila.org}
\date{\today}

\title{Spectra of uniformity}

\begin{document}
\begin{abstract}
We study some limitations and possible occurrences of uniform ultrafilters on ordinals without the axiom of choice. We prove an Easton-like theorem about the possible spectrum of successors of regular cardinals which carry uniform ultrafilters; we also show that this spectrum is not necessarily closed.
\end{abstract}
\maketitle    
\section{Introduction}

The existence of uniform ultrafilters (ultrafilters where all sets have the same cardinality) on an infinite cardinal $\lambda$ is one of the basic consequences of the axiom of choice: simply extend the co-ideal of $[\lambda]^{<\lambda}$ to an ultrafilter using Zorn's Lemma. In fact, one can easily see that this is a consequence of the Ultrafilter Lemma which is known to be strictly weaker than the axiom of choice. In other words, let us denote by $\cU$ the class $\{\lambda\mid\lambda\text{ is an infinite cardinal which carries a uniform ultrafilter}\}$, then $\ZFC$ proves that $\cU$ is equal to the class of all infinite cardinals.

Working only in $\ZF$, the situation can be quite different. It is known to be consistent with $\ZF$ that $\cU$ is empty (see section~\ref{subsection:filters-prelim} for details). In this paper we investigate some of the basic properties of $\cU$ in the theory $\ZF$ + ``Every successor cardinal is regular''. 

The main theorem of this paper is an Easton-like theorem for the restriction of $\mathcal{U}$ to the successors of the regular cardinals. We will also show that $\cU$ is not necessarily closed. Specifically we show that it is consistent that there is no uniform ultrafilter on $\aleph_\omega$, but there are uniform ultrafilters on $\aleph_n$ for all $0<n<\omega$; as well as the opposite scenario where $\aleph_0$ and $\aleph_\omega$ carry uniform ultrafilters, but no $\aleph_n$ for $0<n<\omega$ does.

The construction used for the proof of the Easton-like theorem is somewhat limited, and generally requires very large cardinals to be present in the ground model.  Whether or not these assumptions are entirely necessary remains open. In section 2, we give a brief exposition on symmetric extensions, as well as an historical overview of related results. Sections 3 and 4 are devoted to generalizing previous constructions by Feferman and Jech. Our main theorem, as well as the mentioned consistency results, are proved in section 5. We finish the paper with open questions in section 6.
\subsection*{Acknowledgments}
The authors would like to thank the anonymous referee for their thorough reviewing efforts, as well as for suggesting \autoref{question:referee}, which helped to clarify and improve this paper.
\section{Basic notions}
\subsection{Symmetric extensions}
One of the common methods of constructing models where the axiom of choice fails is using \textit{symmetric extensions}. This is an extension of the method of forcing, and it is generally necessary, since forcing over a model of $\ZFC$ produces a model of $\ZFC$. A symmetric extension is a definable inner model of a generic extension. It contains the ground model and there the axiom of choice can consistently fail.

Let $\PP$ be a notion of forcing, by which we mean a partially ordered class with a maximum element $1$.\footnote{We will explicitly state when we deal with a proper class, though. Any mention of forcing related terminology, unless mentioned explicitly will refer to forcing with a partially ordered set.} We adopt the convention that $q\leq p$ means that $q$ is stronger than $p$, or $q$ extends $p$. Our notation with regards to names is taken from Jech \cite{Jech:ST2003}. Namely, $\dot x$ will denote a $\PP$-name, $\check x$ will denote a canonical name for a ground model set, and $\dot x^G$ denotes the interpretation of $\dot x$ by the filter $G$.

If $\sG$ is a group of automorphisms of $\PP$ we say that $\sF$ is a \textit{normal filter of subgroups over $\sG$} if it is a filter of subgroups which is closed under conjugation. Namely $\sF$ is a non-empty collection of subgroups of $\sG$ closed under finite intersections and supergroups; and if $\pi\in\sG$ and $H\in\sF$, then also $\pi H\pi^{-1}\in\sF$.\footnote{It is enough to require that $\sF$ is a normal filter base, i.e.\ the intersections and conjugations contain an element of the base.}

If $\pi\in\aut(\PP)$, then $\pi$ extends to the $\PP$-names recursively:\[\pi\dot x=\{\tup{\pi p,\pi\dot y}\mid\tup{p,\dot y}\in\dot x\}.\]The Symmetry Lemma presents ties between $\pi$ and the forcing relation (this is Lemma~14.37 in \cite{Jech:ST2003}).
\begin{lemma*}[The Symmetry Lemma]
Let $\PP$ be a forcing notion, $\varphi$ be a formula in the language of forcing with respect to $\PP$, $\dot x$ a $\PP$-name, and $\pi\in\aut(\PP)$. Then \[p\forces\varphi(\dot x)\iff\pi p\forces\varphi(\pi\dot x).\qedhere\]
\end{lemma*}

\begin{definition}
We say that $\tup{\PP,\sG,\sF}$ is a \textit{symmetric system} if $\PP$ is a notion of forcing, $\sG$ is an automorphism group of $\PP$, and $\sF$ is a normal filter of subgroups over $\sG$. We say that the system is \textit{homogeneous} if whenever $p,q\in\PP$, then there is $\pi\in\sG$ such that $\pi p$ and $q$ have a common extension.

We say that the system is \textit{strongly homogeneous} if for every condition $p$ there is a subgroup in $\sF$ which witnesses the homogeneity of $\PP\restriction p=\{q\in\PP\mid q\leq p\}$, the cone below $p$.
\end{definition}

Suppose that $\tup{\PP,\sG,\sF}$ is a symmetric system. If $\dot x$ is a $\PP$-name, we denote by $\sym_\sG(\dot x)$ the group $\{\pi\in\sG\mid\pi\dot x=\dot x\}$, and we say that $\dot x$ is \textit{$\sF$-symmetric} if $\sym_\sG(\dot x)\in\sF$. We recursively define the notion of $\dot x$ being \textit{hereditarily $\sF$-symmetric} if every name which appears in $\dot x$ is hereditarily $\sF$-symmetric and $\dot x$ is $\sF$-symmetric. The class of all hereditarily $\sF$-symmetric names is denoted by $\HS_\sF$. When the context is clear, we omit the subscripts and write $\sym(\dot x)$ and that $\dot x\in\HS$, etc.

Similarly, we say that $A\subseteq\PP$ is a \textit{symmetric subset} if $\{\pi\in\sG\mid\pi\restriction A=\id\}\in\sF$. The proof of the next theorem, and more much, can be found in Chapter 15 of \cite{Jech:ST2003}.
\begin{theorem*}
Let $\tup{\PP,\sG,\sF}$ be a symmetric system, and let $G$ be a $V$-generic filter. Then $M=\HS^G=\{\dot x^G\mid\dot x\in\HS\}$ is a model of $\ZF$ such that $V\subseteq M\subseteq V[G]$ and $M$ is a transitive class of $V[G]$. 
\end{theorem*}
We say that $M$ in the above theorem is a \textit{symmetric extension} of $V$.

Finally, the forcing relation has a relativized version $\forces^\HS$ obtained by relativizing the quantifiers and variables to the class $\HS$. This relation has the same basic properties as the usual forcing relation, with a notable exception that $p\forces^\HS\exists x\,\varphi(x)$ need not imply that there is some $\dot x\in\HS$ such that $p\forces^\HS\varphi(\dot x)$.\footnote{A relatively simple example for this can be found in the Cohen model. In this model, one adds countably many reals $a_n$, and then remembers only the set $A$ of these reals, but not its countability. If $\dot a_n$ denotes the $n$th canonical real and $\dot A$ denotes the canonical name for the set of reals, all of which are in $\HS$, then $1\forces^\HS\exists x(x\in\dot A\land\check 0\in x)$. It is not hard to verify, however, that if $\dot x\in\HS$ and $1\forces^\HS\dot x\in\dot A$, then $\{n\mid 1\not\forces\dot x=\dot a_n\}$ is finite, and by an easy density argument, $1$ does not decide the value of $\check0\in\dot a_n$ for any finitely many reals.}

Moreover, the Symmetry Lemma has a relativized version as well. If $\pi\in\sG$, then \[p\forces^\HS\varphi(\dot x)\iff\pi p\forces^\HS\varphi(\pi\dot x).\]

\subsection{Filters on sets}\label{subsection:filters-prelim}
Let $X$ be a set. Let $\cF\subseteq\power(X)$ be a \textit{filter on $X$}, namely a family of subsets of $X$ closed under finite intersections and upwards inclusion, which does not contain the empty set.
\begin{itemize}
\item We say that $\cF$ is an \textit{ultrafilter} if it is not contained in any larger filter on $X$. Alternatively $\cF$ is an ultrafilter, if for every $A\subseteq X$, either $A\in\cF$ or $X\setminus A\in\cF$.
\item We say that $\cF$ is a \textit{uniform filter} if for all $A\in\cF$, $|A|=|X|$.
\item We say that $\cF$ is a \textit{principal filter} if $\bigcap\cF\in\cF$. If $\cF$ is an ultrafilter, then it is principal if and only if it contains a singleton. A non-principal filter is called \textit{free}.
\item We say that $\cF$ is \textit{$\kappa$-complete} if for all $\gamma<\kappa$ and for all $\{X_\alpha\mid\alpha<\gamma\}\subseteq\cF$, $\bigcap\{X_\alpha\mid\alpha<\gamma\}\in\cF$.
\end{itemize}
\begin{definition}
We say that an $\aleph$ number $\kappa$ is a \textit{measurable cardinal} if there exists a $\kappa$-complete free ultrafilter on $\kappa$. We say that $\kappa$ is a \textit{strongly compact cardinal} if every $\kappa$-complete filter can be extended to a $\kappa$-complete ultrafilter.
\end{definition}
Easily by induction, every filter is $\omega$-complete, and therefore by an easy application of Zorn's Lemma, $\ZFC$ proves that $\aleph_0$ is a strongly compact and a measurable cardinal.\footnote{This leads to the sometimes additional requirement that $\kappa$ is uncountable in the definitions of measurable and strongly compact cardinals. For our purposes, however, it is better to allow $\aleph_0$ to be considered as measurable or strongly compact.}
\begin{theorem*}[Feferman \cite{Feferman:1964}]
It is consistent with $\ZF$ that all ultrafilters on $\aleph_0$ are principal. In other words, it is consistent that $\aleph_0$ is not a measurable cardinal.
\end{theorem*}
This theorem was extended by Andreas Blass to obtain an even stronger result.
\begin{theorem*}[Blass \cite{Blass:1977}]
It is consistent with $\ZF$ that all ultrafilters on all sets are principal.
\end{theorem*}
In his paper Blass sketches the following argument---which he attributes as folklore---to show that Feferman's model ``almost does the job''.
\begin{proposition}\label{prop:first-measurable-is-measurable}
The least $\kappa$ which carries a free ultrafilter is a measurable cardinal and the free ultrafilters on $\kappa$ are uniform. Consequently, in any symmetric extension of $L$, if $\aleph_0$ is not measurable, then all ultrafilters on ordinals are principal.
\end{proposition}
\begin{proof}
Let $\kappa$ be the least ordinal on which there is a free ultrafilter $U$. We will show that $U$ is in fact a $\kappa$-complete measure, and thus $\kappa$ is in fact measurable. Let $\gamma\leq\kappa$ be the least such that there is a partition of $\kappa$, $\{A_\alpha\mid\alpha<\gamma\}$, such that no $A_\alpha$ lies in $U$. Define the map $f(\xi)=\alpha$ if and only if $\xi\in A_\alpha$. Then $f$ is a surjective map of $\kappa$ onto $\gamma$ which maps $U$ to an ultrafilter $U_*$ on $\gamma$ defined as $\{A\subseteq\gamma\mid f^{-1}(A)\in U\}$. If $\gamma<\kappa$, then by the minimality of $\kappa$ it follows that $U_*$ is principal. Therefore there is some $\alpha<\gamma$ such that $\{\alpha\}\in U_*$, which therefore means that $f^{-1}(\{\alpha\})=A_\alpha\in U$. Of course, this is impossible, so $\gamma=\kappa$. So any free ultrafilter on $\kappa$ is $\kappa$-complete, and so uniform. This implies that $\kappa$ is measurable in $L[U]$. In particular, if we work in a symmetric extension of $L$, as no measurable cardinals can be added by forcing, either $\aleph_0$ is measurable, or all ultrafilters on ordinals are principal.
\end{proof}
In the proof above lies the following fact which is worth an explicit mention.
\begin{corollary}
Suppose an infinite ordinal $\kappa$ carries a uniform ultrafilter. If $\aleph_0$ is not measurable, then there is an inner model of $\ZFC$ with a measurable cardinal.\qed
\end{corollary}
\section{Generalization of Feferman's proof}\label{section:gen-fef}
In a recent paper \cite{HerHowKer:2016} by Horst Herrlich, Paul Howard, and Eleftherios Tachtsis, the authors point out that it is open whether or not it is possible that there are ultrafilters on $\omega_1$, but there are no uniform ultrafilters on $\omega_1$. This can be done using a slight generalization of Feferman's argument from \cite{Feferman:1964},\footnote{For a modern approach see Example~15.59 in \cite{Jech:ST2003}.} as we will show in this section. 
\begin{theorem}\label{thm:g-feferman}
Suppose that $V$ is a model of $\GCH$, and $\kappa$ is a regular cardinal. Then there is a symmetric extension of $V$ with the same cardinals where there are no uniform ultrafilters on $\kappa$, and for all $\lambda<\kappa$, $2^\lambda=\lambda^+$.
\end{theorem} 
\begin{proof}
Let $\PP$ be the forcing $\Add(\kappa,\kappa)$. The conditions of $\PP$ are all partial functions $f\colon \kappa \times \kappa \to 2$, such that $|\dom f| <\kappa$, and the ordering is reverse inclusion. We define $\sG$ to be the group of automorphisms $\pi$ with the following property: There exists some $A\subseteq\kappa\times\kappa$ such that $\pi p(\alpha,\beta)=\chi_A(\alpha,\beta)+p(\alpha,\beta)\pmod 2$, where $\chi_A$ is the characteristic function of $A$. In other words, if we think about $p\in\PP$ as a sequence of $0$/$1$ bits indexed by $\kappa\times\kappa$, $\pi$ ``flips'' the values of the bits whose indices are in $A$. We denote by $\pi_A$ the automorphism $\pi$ defined by $A$ as above.

Note that $\sG$ is in fact an abelian group, since $\pi^{-1}=\pi$ for all $\pi\in\sG$. This immediately implies that any filter of subgroups is closed under conjugation. And so we define for $A\subseteq\kappa$, $\fix(A)=\{\pi_B\mid B\cap(A\times\kappa)=\varnothing\}$, and $\sF$ is the filter induced by $\{\fix(A)\mid A\in[\kappa]^{<\kappa}\}$. We work with the symmetric system $\tup{\PP,\sG,\sF}$.

Denote by $\dot x_\alpha$ the name $\{\tup{p,\check\beta}\mid p(\alpha,\beta)=1\}$.

Suppose that $\dot U\in\HS$ is such that $p\forces^\HS``\dot U\text{ is an ultrafilter on }\check\kappa"$. Let $A$ be such that $\fix(A)\leq\sym(\dot U)$, and we may assume without loss of generality that $\dom p\subseteq A\times\kappa$. Let $\alpha\notin A$, and suppose $q\leq p$ is such that $q\forces^\HS\dot x_\alpha\in\dot U$. Find $\beta$ large enough such that $\dom q\subseteq\kappa\times\beta$, and let $X$ be $\{\alpha\}\times(\kappa\setminus\beta)$. The following holds:
\[q\forces^\HS\dot x_\alpha\in\dot U\iff
\pi_X q\forces^\HS\pi_X\dot x_\alpha\in\pi_X\dot U\iff
q\forces^\HS\pi_X\dot x_\alpha\in\dot U.\]
However this implies that $q\forces^\HS\pi_X\dot x_\alpha\cap\dot x_\alpha\in\dot U$, which is a set bounded in $\kappa$, since $\pi_X(\dot x_\alpha)$ is forced to be the symmetric difference between $\kappa\setminus\beta$ and $\dot x_\alpha$. Therefore $q$ forces that $\dot U$ is not uniform. In particular, no extension of $p$ forces that $\dot U$ is uniform, as the same argument works for assuming $\kappa\setminus\dot x_\alpha\in\dot U$, and therefore $p$ must force that $\dot U$ is not uniform.

To see that $2^\lambda=\lambda^+$, note that by the fact that the forcing is $\kappa$-closed, no bounded subsets are added, and $\GCH$ is preserved below $\kappa$.
\end{proof}
\begin{remark}There are two remarks to be made on the proof above:
\begin{enumerate}
\item
The keen eyed observer will notice that actually the model obtained in the above proof also satisfies $\DC_{<\kappa}$,\footnote{$\DC_\lambda$ is the statement that every $\lambda$-closed tree of height $\lambda$ without leaves has a branch; $\DC_{<\kappa}$ abbreviates $\forall\lambda<\kappa,\DC_\lambda$.} as both the forcing and the $\sF$ are $\kappa$-closed, as follows from \cite[Lemma~2.1]{Karagila:2014}.
\item
It is unclear whether or not there are uniform ultrafilters on $\kappa^+$, or any $\lambda>\kappa$, in the above model. The argument in the proof uses the fact that we focus on $\kappa$ in a significant way, and the argument does not go through when considering $\kappa^+$. On the other hand, the homogeneity of the system makes it quite plausible that there are no uniform ultrafilters on $\kappa^+$ without additional hypotheses (e.g.\ large cardinal assumptions).
\end{enumerate}
\end{remark}
The following lemma appears as \cite[Lemma~21.17]{Jech:ST2003}.
\begin{lemma*}[Jech's Lemma]
Let $\kappa$ be measurable in $M$, and let $N$ be a symmetric extension of $M$ (via a symmetric system $\tup{\PP,\sG,\sF}$, where $\PP$ is a complete Boolean algebra, and an $M$-generic filter $G$). If every symmetric subset of $\PP$ has size $<\kappa$, then $\kappa$ is measurable in $N$.
\end{lemma*}
The proof, however, actually proves a stronger lemma as given below.
\begin{lemma}\label{lemma: extending uniform ultrafilters}
Let $\kappa$ be a measurable cardinal, and let $\lambda\geq\kappa$ be a cardinal such that there exists a uniform $\kappa$-complete ultrafilter $U$ on $\lambda$. Suppose that $\tup{\PP,\sG,\sF}$ is a symmetric system such that $\PP$ is a complete Boolean algebra and every symmetric subset of $\PP$ has cardinality $<\kappa$, then $U$ extends uniquely to a $\kappa$-complete uniform ultrafilter on $\lambda$ in the symmetric extension given by the system.
\end{lemma}
\begin{proof}
The only part missing from the proof of Jech's Lemma is the uniformity of the extension. Note that since there is a uniform ultrafilter on $\lambda$ which is $\kappa$-complete, $\cf \lambda \geq \kappa$. This remains true in the symmetric extension. The following is a sketch of the proof of Jech's Lemma. 

Let $\dot X\in\HS$ be a symmetric name for a subset of $\lambda$. There is some $\QQ$, a regular subforcing of $\PP$ of cardinality $\mu<\kappa$, such that $\dot X$ can be seen as a $\QQ$-name. In the generic extension by $\QQ$, there is a unique extension of $U$ to a $\kappa$-complete and uniform ultrafilter on $\lambda$; indeed, in this generic extension every set of ordinals $Y$ is a union of at most $\mu$ sets from the ground model, $\{A_i \mid i < \mu\}$. So $Y$ belongs to the extension of $U$ if and only if one of them belongs to $U$.

Let $\widetilde{U}$ be the $\PP$-name for the union of the unique extensions of $U$ in each generic extension by a regular subforcing of $\PP$ generated by a symmetric subset of $\PP$. 

It is routine to verify that $\widetilde{U}$ is stable under all the automorphisms in $\sG$ and is indeed in $\HS$. Moreover as every set of ordinals in the symmetric extension was introduced by a small subforcing, we get that $1\forces^\HS\widetilde{U}\text{ is an ultrafilter on }\check\lambda$.

Finally, if $\dot X$ is a name for a set in $\widetilde{U}$, then there is some intermediate extension where $\dot X$ is in the unique extension of $U$ to a uniform ultrafilter on $\lambda$, and therefore $\dot X$ is a name for a set of size $\lambda$, so indeed $1\forces^\HS\widetilde{U}\text{ is a uniform ultrafilter on }\check\lambda$.
\end{proof}
In turn, this brings us to the following corollary.
\begin{corollary}
Under the notation and conditions on $\kappa$ and $\lambda$ of the previous lemma, if $\mu<\kappa$ is some regular cardinal, taking the symmetric extension given by the generalized Feferman construction for $\mu$, any $\kappa$-complete uniform ultrafilter on $\lambda$ extends to a $\kappa$-complete uniform ultrafilter on $\lambda$ in the symmetric extension.\qed
\end{corollary}
\section{Symmetric collapses}
\begin{definition}
Let $\kappa\leq\lambda$ be two infinite cardinals. The \textit{symmetric collapse} is the symmetric system $\tup{\PP,\sG,\sF}$ defined as follows:
\begin{itemize}
\item $\PP=\Col(\kappa,<\lambda)$, so a condition in $\PP$ is a partial function $p$ with domain $\{\tup{\alpha,\beta}\mid\kappa<\alpha<\lambda,\beta<\kappa\}$ such that $p(\alpha,\beta)<\alpha$ for all $\alpha$ and $\beta$, $\supp(p)=\{\alpha<\lambda\mid\exists\beta,\ \tup{\alpha,\beta}\in\dom p\}$ is bounded below $\lambda$ and $|p|<\kappa$.
\item $\sG$ is the group of automorphisms $\pi$ such that there is a sequence of permutations $\vec\pi=\tup{\pi_\alpha\mid\kappa<\alpha<\lambda}$ such that $\pi_\alpha$ is a permutation of $\alpha$ satisfying $\pi p(\alpha,\beta)=\pi_\alpha(p(\alpha,\beta))$ (note that $p(\alpha,\beta)$ is an ordinal below $\alpha$).
\item $\sF$ is the normal filter of subgroups generated by $\fix(E)$ for bounded $E\subseteq\lambda$, where $\fix(E)$ is the group $\{\pi\mid\forall\alpha\in E,\pi p(\alpha,\beta)=p(\alpha,\beta)\}$, i.e.\ if $\pi$ is induced by $\vec\pi$, then $\pi\in\fix(E)$ if and only if $\pi_\alpha=\id$ for all $\alpha\in E$.
\end{itemize}
\end{definition}
\begin{theorem}[Folklore]\label{thm:folk-measure-preserving}
Assume $\GCH$ holds, and let $\kappa\leq\lambda$ be two cardinals. The symmetric extension given by the symmetric collapse satisfies the following properties:
\begin{enumerate}
\item The symmetric system is strongly homogeneous.
\item $\lambda=\kappa^+$, and if $\lambda$ was regular, then it remains regular.
\item If $\mu<\kappa$, then no new subsets of $\mu$ were added.
\item Every symmetric subset of $\PP$ has cardinality $<\lambda$, in particular if $\lambda$ was measurable it remains measurable.\qed
\end{enumerate}
\end{theorem}
Working over $L$, the case of $\kappa=\aleph_0$ and $\lambda$ a limit cardinal was studied by John Truss in \cite{Truss:1974} as an extension of the case where $\lambda$ was strongly inaccessible, which was studied by Robert Solovay in \cite{Solovay:1970}.

It is well known that the existence of a free ultrafilter on $\aleph_0$ implies that there exists a non-measurable set of reals. Solovay's model where one takes the symmetric collapse of an inaccessible is a model where all sets of reals are Lebesgue measurable. But we can prove a more general theorem on symmetric collapses.
\begin{theorem}\label{thm:symm-coll-and-uf}
Suppose that $\kappa$ is regular and $\lambda\geq\kappa$. Let $M$ be the symmetric extension obtained by the symmetric collapse $\Col(\kappa,<\lambda)$. In $M$ there are no uniform ultrafilters on $\kappa$.
\end{theorem}
\begin{proof}
The argument is similar to the generalized Feferman model, as described in section~\ref{section:gen-fef}. Suppose that $\dot U\in\HS$ is such that $p_0\forces\dot U\text{ is a uniform ultrafilter on }\check\kappa$. Let $\alpha<\lambda$ be large enough such that $p_0$ is bounded below $\alpha$, i.e.\ only ordinals below $\alpha$ appear in the domain of $p_0$, and $\fix(\alpha)\leq\sym(\dot U)$, and take $\beta>\alpha$. Let $\dot x=\{\tup{p,\check\xi}\mid p(\beta,\xi)\text{ is even}\}$, easily $\dot x$ is symmetric. Suppose that $q\leq p_0$ decides the truth value of $\dot x\in\dot U$.

Consider the involution $\pi$ which only permutes the $\beta$th coordinate in the following way: find $\gamma$ large enough such that $q$ does not mention any ordinal above $\gamma$ in its $\beta$th coordinate; then partition the interval $(\gamma,\beta)$ into pairs $\{\zeta,\zeta'\}$ where exactly one of these is even; finally, define $\pi p(\beta,\xi)=\zeta'$ if and only if $p(\beta,\xi)=\zeta$ when $\{\zeta,\zeta'\}$ is in the above partition (everywhere else $\pi$ is the identity). In other words, define an involution which switches the parity of all large enough ordinals, only on the $\beta$th coordinate of the condition such that $\pi q=q$.

This readily implies that $\pi\dot U=\dot U$, and that $\pi q=q$. From the definition of $\dot x$, and as in the proof of \autoref{thm:g-feferman}, we have that \[q\forces|\dot x\cap\pi\dot x|,|\check\kappa\setminus(\dot x\cup\pi\dot x)|<\check\kappa,\]
and therefore $q$ cannot force that $\dot U$ is uniform, as one of these has to be in $\dot U$.
\end{proof}
\begin{corollary}
Jech's model, obtained by taking the symmetric collapse where $\kappa=\aleph_0$ and $\lambda$ is a measurable cardinal, satisfies that $\aleph_0$ is not measurable, but $\aleph_1$ is measurable.\qed
\end{corollary}
\begin{remark}\label{rem:collapse-and-feferman}
Note that if $\lambda=\kappa$ and $\kappa$ is regular, then the symmetric collapse is exactly the generalized Feferman construction for $\kappa$. If $\lambda=\kappa^+$, however, the symmetric collapse is a mild variant of the generalized Feferman model, as $\Col(\kappa,\kappa^+)$ is equivalent to $\Add(\kappa,\kappa^+)$.
\end{remark}
\section{Uniformity spectra}
\subsection{Singular limitations}
\begin{theorem}[$\ZF$]\label{thm:cofinality-reflection}
Suppose that $\lambda$ is singular and $\cf(\lambda)=\kappa$. If there is a uniform ultrafilter on $\lambda$, then there is a uniform ultrafilter on $\kappa$.
\end{theorem}
\begin{proof}
Let $U$ be a uniform ultrafilter on $\lambda$, fix a cofinal set $\{\lambda_\alpha\mid\alpha<\kappa\}$, and let $X_\alpha$ be the interval $[\lambda_\alpha,\lambda_{\alpha+1})$.\footnote{It does not matter if $\lambda_0=0$, but it is convenient to assume that it is.} Define $U_*$ as follows,\[A\in U_*\iff A\subseteq\kappa\text{ and }\bigcup\{X_\alpha\mid\alpha\in A\}\in U.\]
The claim that $U_*$ is an ultrafilter on $\kappa$ follows from the fact that $U$ is an ultrafilter, and that the complement of a union of $X_\alpha$'s is itself a union of $X_\alpha$'s. To see that $U_*$ is uniform, note that $\kappa$ is regular---$\cf(\lambda)$ is always regular, even without assuming choice---so every small set $A$ is bounded, but then $\bigcup\{X_\alpha\mid\alpha\in A\}$ is a bounded subset of $\lambda$ and therefore has small cardinality and so is not in $U$, so $A\notin U_*$.
\end{proof}
\begin{remark}
It seems tempting to claim that the above proof should work for every unbounded set, thus it is enough to require that there is an increasing and cofinal function from $\kappa$ to $\lambda$. However, if $\kappa$ happened to be singular (note that $\cf(\kappa)=\cf(\lambda)$), then the claim that a small set is bounded is no longer true, so we cannot obtain uniformity. So a uniform ultrafilter on $\aleph_{\omega_\omega}$ implies a uniform ultrafilter must exist on $\aleph_0$, but not necessarily---it seems---on $\aleph_\omega$ (see \autoref{question:referee}).
\end{remark}
\begin{corollary}
Let $\lambda$ be a $\lambda^{+\omega}$-strongly compact cardinal. Taking the symmetric collapse with $\kappa=\aleph_0$ produces a model where for $0<n<\omega$, there is a uniform ultrafilter on $\aleph_n$, but there is no uniform ultrafilter on $\aleph_\omega$.
\end{corollary}
\begin{proof}
By \autoref{thm:symm-coll-and-uf} there are no uniform ultrafilters on $\aleph_0$,\footnote{Note that in the case of an ultrafilter on $\aleph_0$, uniform is equivalent to free.} and so by \autoref{thm:cofinality-reflection}, there is no uniform ultrafilter on $\aleph_\omega$. However by strong compactness, $\lambda^{+n}$ carries a $\lambda$-complete uniform ultrafilter, and by \autoref{lemma: extending uniform ultrafilters} and \autoref{thm:folk-measure-preserving} these extend in the symmetric extension.
\end{proof}
\begin{remark}
If $\lambda$ happens to be fully strongly compact, the resulting model is such that every successor cardinal is regular and $\mu$ carries a uniform ultrafilter if and only if $\cf(\mu)$ is uncountable, by the theorems and lemmas proved so far.
\end{remark}
\begin{theorem}\label{thm:integration}
Let $\lambda$ be singular and let $A\subseteq\lambda$ be an unbounded set. Assume that $\tup{U_\kappa\mid\kappa\in A}$ is a sequence such that $U_\kappa$ is a uniform ultrafilter on $\kappa$. If $U$ is a uniform ultrafilter on $\otp(A)$, then there is a uniform ultrafilter on $\lambda$.
\end{theorem}
\begin{proof}
Define $U^*$ to be the integration of $\tup{U_\kappa\mid\kappa\in A}$ with respect to $U$,\[X\in U^*\iff\{\alpha\mid X\cap\kappa_\alpha\in U_\alpha\}\in U.\qedhere\]
\end{proof}
It is interesting to note that without choice, it is not reasonable to assume that just because there is a uniform ultrafilter on each $\kappa$ in $A$, there is also a choice sequence of uniform ultrafilters. The next theorem, however, shows that the existence of a uniform ultrafilter on $\lambda$ need not imply the existence of uniform ultrafilters on smaller cardinals other than $\cf(\lambda)$. Combining the previous theorem with the next one also leads us quite naturally to \autoref{question:non-closure}.
\begin{theorem}
Assuming the existence of countably many measurable cardinals, it is consistent that there is a uniform ultrafilter on $\aleph_\omega$ but for all $0<n<\omega$, there are no uniform ultrafilters on $\aleph_n$.
\end{theorem}
\begin{proof}
Without loss of generality, assume that $2^{\aleph_0}=\aleph_1$, otherwise we can force this while preserving the measurable cardinals. Let $\kappa_0=\aleph_1$ and let $\kappa_n$ for $n>0$ be a sequence of measurable cardinals, with $\kappa_\omega=\sup\{\kappa_n\mid n<\omega\}$. Let $U_0$ be a uniform ultrafilter on $\aleph_0$ and $U_n$ a fixed normal measure on $\kappa_n$ for $n>0$. The ultrafilter $U$ obtained by integrating the $U_n$'s over $U_0$ is a uniform ultrafilter on $\kappa_\omega$, i.e., \[A\in U\iff\{n<\omega\mid A\cap\kappa_{n+1}\in U_{n+1}\}\in U_0.\]
Consider the symmetric extension obtained by taking the finite support product of the symmetric collapses $\Col(\kappa_n,<\kappa_{n+1})$. In the resulting model $M$ the following hold:
\begin{enumerate}
\item If $A$ is a set of ordinals, then there is some $n$ such that $A$ was introduced by collapses below $\kappa_n$.
\item $2^{\aleph_0}=\aleph_1$, and moreover no new reals are added. In particular $U_0$ is still an ultrafilter on $\aleph_0$.
\item For $n>0$, $\aleph_n=\kappa_{n-1}$, and there are no uniform ultrafilters on $\aleph_n$.
\item $U$ extends to an ultrafilter.
\end{enumerate}
To see that (1) is true, note that if $\dot A$ is a name for a set of ordinals, then $\dot A^*$ defined as $\{\tup{p,\check\xi}\mid p\forces\check\xi\in\dot A\}$ is such that $\dot A^*\in\HS$ and $1\forces\dot A=\dot A^*$ (it follows from the Symmetry Lemma that $\dot A$ and $\dot A^*$ are fixed by the same automorphisms). We will say that $\dot A$ is a \textit{nice name} if $\dot A=\dot A^*$, and we can always assume therefore names for sets of ordinals are nice. By the strong homogeneity of the system we obtain that if $n$ is such that $\sym(\dot A)$ contains all the permutations which are the identity on $\prod_{m<n}\Col(\kappa_m,<\kappa_{m+1})$, then every condition which appears in $\dot A$ can be restricted to the product of the first $n$ symmetric collapses. This readily implies (1) and also gives (2), since no finitely many collapses add any subsets of $\omega$. It also follows that for $n>0$, $M\models\aleph_n=\kappa_{n-1}$: note that this is the case in the full generic extension, and that the collapsing functions of any ordinal in the intervals $(\kappa_n,\kappa_{n+1})$ have a symmetric name. Of course, this means that in $M$, $\kappa_\omega=\aleph_\omega$.

The fact that there are no uniform ultrafilters on $\aleph_n$ for $n>0$ follows from \autoref{thm:symm-coll-and-uf}. Finally, $U$ extends to an ultrafilter since given any $n$, looking at the intermediate model obtained by the product of the first $n$ systems, $U$ can be extended to an ultrafilter there. This is because $U_0$ and a tail of the measures $U_n$ are preserved by \autoref{lemma: extending uniform ultrafilters} and \autoref{thm:folk-measure-preserving}. Therefore in order to decide whether or not $A\subseteq\kappa_\omega$ is in $U$, we only need to ask whether or not it entered the extension of $U$ in some bounded part of the product.

Formally, for all $n$, let $\dot U_{*(n)}$ denote the canonical name for the unique extension of $U$ in the model obtained by the symmetric product $\prod_{m<n}\Col(\kappa_m,<\kappa_{m+1})$. This name is obtained by looking at the canonical extension of each $U_k$ for $k\geq n$, and their integration using $U_0$. Finally, define $\dot U_*$ as follows,
\[\dot U_*=\left\{\tup{p,\dot A}\middd\begin{array}{l}
\dot A\in\HS\text{ is a nice name for a subset of }\kappa_\omega\text,\\
\exists n\text{ such that:}\\
\fix(n)\leq\sym(\dot A),\text{ and }\\
p\restriction n\forces^\HS_n\dot A\in\dot U_{*(n)}.
\end{array}\right\}.\]
It is an exercise in verifying definitions to see that in fact all the automorphisms preserve the name $\dot U_*$. Therefore it is in fact in $\HS$, and moreover it follows that $1\forces^\HS\dot U_*\text{ is a uniform ultrafilter on }\check\kappa_\omega$.
\end{proof}
\subsection{Uniformity spectra on successor cardinals}
The rest of this section is devoted to proving the main theorem of interest.

\begin{theorem}\label{thm:easton-like}
Assume $\GCH$ and that there is a proper class of strongly compact cardinals. Let $F\colon\Ord\cup\{-1\}\to2$ be a class function such that $F(\alpha)=0$ for every infinite limit ordinal $\alpha$. Then there is a class symmetric extension $M_F$ satisfying:
\begin{enumerate}
\item Every successor cardinal is regular and successors of singular cardinals in $M_F$ are computed the same as in the ground model.
\item There exists a uniform ultrafilter on $\aleph_{\alpha+1}$ if and only if $F(\alpha)=1$.
\item If $\lambda$ is a singular cardinal such that \[\sup(\{\mu^+ < \lambda \mid \mu^+\in \cU\})=\lambda\] and $\cf(\lambda)\in\cU$, then there is a uniform ultrafilter on $\lambda$.
\end{enumerate} 
\end{theorem}
\begin{proof}
We define by recursion the following continuous sequence of cardinals $\kappa_\alpha$, such that in $M_F$, $\aleph_\alpha=\kappa_\alpha$:
\begin{enumerate}
\item $\kappa_0=\aleph_0$.
\item For a limit $\alpha$, $\kappa_\alpha=\sup\{\kappa_\beta\mid\beta<\alpha\}$, and $\kappa_{\alpha+1}=\kappa_\alpha^+$.
\item If $\alpha=\beta+2$ with $F(\beta+1)=1$ and $F(\beta)=0$, then $\kappa_\alpha$ is the least strongly compact cardinal larger than $\kappa_{\beta+1}$.
\item If $\alpha=\beta+2$ and the previous case does not hold, then $\kappa_\alpha=\kappa_{\beta+1}^+$.
\end{enumerate}

For every $\alpha$, $\tup{\QQ_\alpha,\sG_\alpha,\sF_\alpha}$ is the symmetric collapse $\Col(\kappa_\alpha,<\kappa_{\alpha+1})$ when $\alpha$ is a successor ordinal (or $0$) and $F(\alpha-1)=0$, and the trivial symmetric system otherwise. Note that in the case where $\kappa_{\alpha+1}=\kappa_\alpha^+$, the forcing is just $\Add(\kappa_\alpha,\kappa_\alpha^+)$. So if, for example, $F(\alpha)=0$ for all $\alpha$, the symmetric construction is the mild variant of the generalized Feferman construction mentioned in \autoref{rem:collapse-and-feferman} (in particular no cardinals are collapsed in that case as we assume $\GCH$).

We now take the Easton support product of the symmetric systems as described in \cite{Karagila:2014}, and let $\cS$ be the product. Let $M_F$ be the symmetric extension by the symmetric system $\cS$. Combining the theorems from the previous sections, and the fact that each of the forcings is sufficiently closed, we sketch the proof that $M_F$ satisfies (1)--(3).

Working in $M_F$ we show that (1) holds, by proving inductively that $\aleph_\alpha=\kappa_\alpha$. By definition, $\aleph_0=\kappa_0$. The limit case is trivial as is the successor of limit case. Suppose that $\aleph_\alpha=\kappa_\alpha$, then either the $\alpha$th symmetric system was trivial and $\kappa_{\alpha+1}=(\kappa_\alpha^+)^V=(\kappa_\alpha^+)^{M_F}$, or $F(\alpha)=1$. In this case, $\kappa_{\alpha+1}$ is a strongly compact cardinal, and we force with $\Col(\kappa_\alpha,<\kappa_{\alpha+1})$. Therefore $M_F\models\kappa_\alpha^+=\aleph_{\alpha+1}=\kappa_{\alpha+1}$. Since $\kappa_{\alpha+1}$ is regular in the full generic extension and regularity of a cardinal is downwards absolute, $\kappa_{\alpha+1}$ is regular in the symmetric extension $M_F$ as well. 

Next we show that (2) holds in $M_F$. Say that $\alpha$ is a \textit{flip point} if it is a successor ordinal or zero such that $F(\alpha)=1$ and $F(\alpha-1)=0$. Note that for $\alpha \geq 0$, being a flip point coincides with $\kappa_{\alpha+1}$ being a strongly compact cardinal in the ground model. Moreover, if $\alpha > \omega$, $F(\alpha)=1$ and $\alpha$ is not a flip point, then there is a flip point $\beta$ such that $\alpha=\beta+n$ for some $n<\omega$.

First, using Easton's lemma, in the full generic extension, any subset of $\kappa_\alpha$ is added by the first $\alpha+1$ steps of the product. Thus, $\ZFC$ holds in the generic extension and $\ZF$ holds in the symmetric extension. 

Let us verify that in the symmetric extension there is a uniform ultrafilter on $\aleph_{\alpha + 1}$ if and only if $F(\alpha) = 1$. Let us deal first with the case $F(\alpha) = 0$. In this case, the system $\cS$ splits into a product of $\cS'$ and $\cS_{\alpha+1} = \tup{\QQ_{\alpha+1},\sG_{\alpha+1},\sF_{\alpha+1}}$. Let $M'$ be the symmetric extension by $\cS'$. By extending $M'$ using the system $\cS_{\alpha+1}$ we clearly obtain the model $M_F$. Thus, by applying \autoref{thm:symm-coll-and-uf} and \autoref{rem:collapse-and-feferman} in $M'$, we conclude that there is no uniform ultrafilter over $\aleph_{\alpha+1}$ in $M_F$. 

Let us deal now with the case $F(\alpha) = 1$. For every ordinal $\gamma$, the symmetric system $\cS$ splits into two parts: the part below $\gamma$, which we denote by $\cS \restriction\gamma$, and the rest of the forcing, which we denote by $\cS\restriction [\gamma, \infty)$. Let $\beta\in\{-1\}\cup\alpha$ be maximal such that $F(\beta)=0$, if there is one, or $\beta=-2$ otherwise.\footnote{Note that $\alpha-\beta$ is finite since for every infinite limit ordinal $\delta$, $F(\delta) = 0$.} Since $\beta+1\leq\alpha$ is either a flip point or $-1$, $\kappa_{\beta+2}$ is strongly compact. Therefore in the ground model, for each regular cardinal $\mu\geq\kappa_{\beta+2}$, there is a uniform $\kappa_{\beta+2}$-complete ultrafilter on $\mu$, $U$. By \autoref{lemma: extending uniform ultrafilters}, $U$ uniquely extends to a uniform ultrafilter in the symmetric extension by $\cS\restriction\beta+2$. Let $\widetilde{U}$ be the unique extension.

For all $\beta<\gamma\leq\alpha$, $F(\gamma)=1$. Therefore $\QQ_{\gamma+1}$ is trivial for all such $\gamma$. In other words, $\cS\restriction\beta+2$ is equivalent to $\cS\restriction\alpha+2$, so $\widetilde{U}$ is still a uniform ultrafilter in the symmetric extension by $\cS\restriction\alpha+2$. The remainder, $\cS\restriction[\alpha+2,\infty)$, does not add any new subsets of $\kappa_{\alpha+1}$, so $\widetilde{U}$ is a uniform ultrafilter on $\aleph_{\alpha+1}$ in $M_F$ as well.

Finally, we show that (3) holds. Suppose that $\lambda$ is a singular cardinal such that there exists an increasing sequence $\tup{\lambda_\alpha\mid\alpha<\cf(\lambda)}$ of successor cardinals whose supremum is $\lambda$, such that $\cf(\lambda),\lambda_\alpha\in\cU$ for all $\alpha<\cf(\lambda)$. Since $\lambda_\alpha$ is a successor cardinal, it is equal to some $\kappa_{\beta_\alpha+1}$. In the ground model, using the axiom of choice, let $U_\alpha$ be a uniform ultrafilter on $\kappa_{\beta_\alpha+1}$ whose closure is some $\kappa\leq\kappa_{\beta_\alpha+1}$ which is a flip point, or any uniform ultrafilters if there are no flip points less or equal than $\kappa_{\beta_\alpha+1}$. By the above, each $U_\alpha$ extends uniquely to some $\widetilde{U}_\alpha$ in $M_F$, and by the assumption that $\cf(\lambda)\in\cU$, we can apply \autoref{thm:integration} to obtain a uniform ultrafilter on $\lambda$ in $M_F$, as wanted.
\end{proof}

\section{Open problems}
In this paper we only dealt with the situation where every successor cardinal is regular. The following questions are left open in the same context.
\begin{question}
Is it consistent for $\aleph_{\omega+1}$ to be the least element of $\cU$? More generally, what behavior is consistent at successors of singular cardinals?
\end{question}
\begin{question}\label{question:non-closure}
Is it consistent for a singular cardinal, and specifically $\aleph_\omega$, to be the least cardinal \textit{not} in $\cU$? 
\end{question}
\begin{question}\label{question:referee}
Assume there is a uniform ultrafilter on $\aleph_{\omega_\omega}$, does that imply there is a uniform ultrafilter on $\aleph_\omega$? Or more generally, if $\lambda>\cf(\lambda)$ carries a uniform ultrafilter, does that imply that any other singular cardinal with the same cofinality carries a uniform ultrafilter?
\end{question}
\begin{question}
What are the exact limitations on $\cU$ in $\ZF$? Is \autoref{thm:cofinality-reflection} the only provable limitation?
\end{question}

One might argue that in some cases a proper class of strongly compact cardinals is a bit excessive. That is indeed correct. If we only want that $\cU$ is an initial segment (or even empty as happens in Feferman's original construction), then clearly one needs no large cardinals at all as this is can be obtained by taking an Easton support product of the symmetric systems described in section~\ref{section:gen-fef}. Even some less trivial patterns can be obtained from weaker hypotheses, e.g.\ if we define the function $F\colon \{-1\}\cup \Ord \to \{0,1\}$ by $F(\alpha + 2n) = 0, F(\alpha+2n+1)=1$ for all limit ordinal $\alpha$ (including $\alpha = 0$) and every natural number $n$ and $F(-1)=0$, then a model satisfying $\aleph_{\alpha + 1} \in \cU \iff F(\alpha) = 1$ can be obtained using just a proper class of measurable cardinals using the same construction as \autoref{thm:easton-like}. One can also allow for longer blocks of cardinals to have uniform ultrafilters by first ensuring there are uniform ultrafilters on the cofinalities of the singulars in the block, and then simply take a gap before the next symmetric collapse. For example, starting with one strongly compact cardinal and symmetrically collapsing it to be $\aleph_4$ will ensure the block of cardinals $[\aleph_4,\aleph_{\omega_3})$ will all carry uniform ultrafilters.

It is very unclear, however, how low these constructions can go. Some natural questions arise from these considerations.
\begin{question}
Is it consistent that $\kappa$ does not carry a uniform ultrafilter, $\kappa^+$ does, but $\kappa^+$ is not measurable, and is it possible without using large cardinals? In particular, is it consistent that $\aleph_0$ is the only measurable cardinal, while $\aleph_1\notin\cU$ and $\aleph_2\in\cU$?
\end{question}
\begin{question}
What is the large cardinal strength of having $\kappa>\aleph_0$ the least element of $\cU$, with $\kappa^+$ having a $\kappa$-complete, $\kappa^+$-incomplete, uniform ultrafilter?
\end{question}

Of course, once we allow successor cardinals to be singular (e.g., models where $\AD$ holds) the techniques above can no longer produce such results and a different approach is needed. There are many natural questions to ask in these contexts, and some trivial solutions. For example, it was shown to be consistent that the least regular cardinal is the least measurable cardinal and that it can be $\aleph_{\omega+1}$ \cite{AptDimKoep:2014}; from the existence of such a model the first question above is answered positively.\footnote{If there are any free ultrafilters on $\aleph_0$, by using the original Feferman construction one kills these ultrafilters, and as it is a countable forcing it does not collapse cardinals or destroy measurability. The result would be the least cardinal carrying a uniform ultrafilter is also the least regular cardinal which is exactly $\aleph_{\omega+1}$. It should be added that it is unclear whether there are free ultrafilters on $\aleph_0$ in that model, and it seems unlikely that there are any.}
\bibliographystyle{amsplain}

\begin{thebibliography}{1}

\bibitem{AptDimKoep:2014}
Arthur~W. Apter, Ioanna~M. Dimitriou, and Peter Koepke, \emph{The first
  measurable cardinal can be the first uncountable regular cardinal at any
  successor height}, MLQ Math. Log. Q. \textbf{60} (2014), no.~6, 471--486.
  \MR{3274975}

\bibitem{Blass:1977}
Andreas Blass, \emph{A model without ultrafilters}, Bull. Acad. Polon. Sci.
  S\'er. Sci. Math. Astronom. Phys. \textbf{25} (1977), no.~4, 329--331.
  \MR{0476510}

\bibitem{Feferman:1964}
S.~Feferman, \emph{Some applications of the notions of forcing and generic
  sets}, Fund. Math. \textbf{56} (1964/1965), 325--345. \MR{0176925}

\bibitem{HerHowKer:2016}
Horst Herrlich, Paul Howard, and Kyriakos Keremedis, \emph{On preimages of
  ultrafilters in {ZF}}, Comment. Math. Univ. Carolin. \textbf{57} (2016),
  no.~2, 241--252. \MR{3513447}

\bibitem{Jech:ST2003}
Thomas Jech, \emph{Set theory. {T}he third millennium edition, revised and
  expanded}, Springer Monographs in Mathematics, Springer-Verlag, Berlin, 2003.
  \MR{1940513}

\bibitem{Karagila:2014}
Asaf Karagila, \emph{Embedding orders into the cardinals with
  {$\mathsf{DC}_\kappa$}}, Fund. Math. \textbf{226} (2014), no.~2, 143--156.
  \MR{3224118}

\bibitem{Solovay:1970}
Robert~M. Solovay, \emph{A model of set-theory in which every set of reals is
  {L}ebesgue measurable}, Ann. of Math. (2) \textbf{92} (1970), 1--56.
  \MR{0265151}

\bibitem{Truss:1974}
John Truss, \emph{Models of set theory containing many perfect sets}, Ann.
  Math. Logic \textbf{7} (1974), 197--219. \MR{0369068}

\end{thebibliography}
\providecommand{\bysame}{\leavevmode\hbox to3em{\hrulefill}\thinspace}
\providecommand{\MR}{\relax\ifhmode\unskip\space\fi MR }
\providecommand{\MRhref}[2]{%
  \href{http://www.ams.org/mathscinet-getitem?mr=#1}{#2}
}
\providecommand{\href}[2]{#2}

\end{document}